\documentclass[11pt,twoside]{amsart}
\usepackage[latin1]{inputenc}
\usepackage[OT1]{fontenc}
\usepackage{amsmath}
\usepackage{amsthm}
\usepackage{amssymb}
\usepackage[all]{xy}

\newtheorem{thm}{Theorem}[section]

\newtheorem{prop}[thm]{Proposition}

\newtheorem{lem}[thm]{Lemma}

\numberwithin{equation}{section}

\theoremstyle{definition}

\newtheorem{remark}[thm]{Remark}

\newcommand{\qqed}{\hspace*{\fill}$\Box$}

\newcommand{\im}{\operatorname{im}}
\newcommand{\Db}{{\rm D}^{\rm b}}

\newcommand{\Aut}{{\rm Aut}}

\newcommand{\Pic}{{\rm Pic}}

\newcommand{\Hom}{{\rm Hom}}
\newcommand{\Spec}{{\rm Spec}}

\newcommand{\Hilb}{{\rm Hilb}}

\newcommand{\id}{{\rm id}}

\newcommand{\cal}{\mathcal}
\newcommand{\ka}{{\cal A}}

\newcommand{\kk}{{\cal K}}
\newcommand{\kl}{{\cal L}}

\newcommand{\km}{{\cal M}}
\newcommand{\ko}{{\cal O}}

\newcommand{\kx}{{\cal X}}

\newcommand{\IA}{\mathbb{A}}
\newcommand{\IC}{\mathbb{C}}

\newcommand{\IK}{\mathbb{K}}
\newcommand{\IL}{\mathbb{L}}

\newcommand{\IP}{\mathbb{P}}
\newcommand{\IQ}{\mathbb{Q}}

\newcommand{\IZ}{\mathbb{Z}}

\renewcommand{\to}{\xymatrix@1@=15pt{\ar[r]&}}
\renewcommand{\rightarrow}{\xymatrix@1@=15pt{\ar[r]&}}
\renewcommand{\mapsto}{\xymatrix@1@=15pt{\ar@{|->}[r]&}}
\renewcommand{\twoheadrightarrow}{\xymatrix@1@=15pt{\ar@{->>}[r]&}}
\renewcommand{\hookrightarrow}{\xymatrix@1@=15pt{\ar@{^(->}[r]&}}
\newcommand{\congpf}{\xymatrix@1@=15pt{\ar[r]^-\sim&}}
\renewcommand{\cong}{\simeq}

\begin{document}

\title[Derived equivalent conjugate K3 surfaces]{Derived equivalent conjugate K3 surfaces}
\author[P.\ Sosna]{Pawel Sosna}

\address{Dipartimento di Matematica ``F.\ Enriques'', Universit\`a degli Studi di Milano, Via Cesare Saldini 50,
20133 Milano, Italy}
\email{pawel.sosna@guest.unimi.it}

\begin{abstract} \noindent
We show that there exist a complex projective K3 surface $X$ and an automorphism $\sigma \in \Aut(\IC)$
such that the conjugate K3 surface $X^\sigma$ is a non-isomorphic Fourier--Mukai partner of
$X$.

\vspace{-2mm}\end{abstract}
\maketitle

\maketitle
\let\thefootnote\relax\footnotetext{This work was supported by the SFB/TR 45 `Periods,
Moduli Spaces and Arithmetic of Algebraic Varieties' of the DFG
(German Research Foundation)}

\section{Introduction}
By definition two complex projective K3 surfaces $X$ and $Y$ are called Fourier--Mukai partners (or derived equivalent) if there exists a $\mathbb{C}$-linear  exact equivalence between their derived categories of coherent sheaves. Due to results by Mukai (\cite{Mukai}) and Orlov (\cite{Orlov}) we have geometric and cohomological criteria for K3 surfaces to be derived equivalent. In particular, it follows that any given K3 surface has only finitely many non-isomorphic Fourier--Mukai partners. It is also possible to view derived equivalent K3 surfaces as elements in the orbit of the action of a certain discrete group on the moduli space of (generalized) K3 surfaces (see e.g.\,\cite{Huy}).

There is a different group acting on the moduli space, namely the group of automorphisms of the complex numbers. The action is defined as follows.
Consider a complex projective K3 surface $X$ and $\sigma \in \Aut(\IC)$. Then we define the \emph{conjugate K3 surface} $X^\sigma$ by the fibre product
\[\begin{xy}
\xymatrix{ X^\sigma \ar[r] \ar[d] & X \ar[d] \\
	   \Spec(\mathbb{C}) \ar[r]^{\sigma^*} & \Spec(\mathbb{C}).}
  \end{xy}
\]
Clearly this definition can be applied to any complex variety. In general the process of conjugation can change the homotopy type as was shown in \cite{Serre}, where Serre constructs two conjugate complex projective varieties with different fundamental groups. In our special case the change in geometry caused by conjugation is more subtle.

The K3 surfaces $X$ and $X^\sigma$ will in general be non-isomorphic as schemes over $\IC$, but clearly they are isomorphic over $\IK=\IQ$ and thus there is a $\IQ$-linear exact equivalence $\Db(X)\cong_\IQ \Db(X^\sigma)$. An obvious question is whether one can find examples where there is a $\IC$-linear exact equivalence as well, without $X$ and $X^\sigma$ being isomorphic over $\IC$.

Thus, we would like to understand the connection between the actions of $\Aut(\IC)$ and the above mentioned discrete group on the moduli space of projective K3 surfaces.

The main result is the following theorem which shows that the orbits of the two group actions can intersect in more than one point:

\medskip

\noindent{\bf Theorem} {\it There exist a complex projective K3 surface $X$ and an automorphism $\sigma \in \Aut(\IC)$ such that the conjugate K3 surface $X^\sigma$ is a non-isomorphic complex K3 surface, but there exists a $\IC$-linear equivalence $\Db(X)\congpf \Db(X^\sigma)$.}

\medskip 

\noindent
The basic idea of the proof is to find a curve $C$ in the moduli space of K3 surfaces over $\overline{\IQ}$ which is invariant under a so called Mukai involution, which is a map sending a K3 surface $X$ to a certain moduli space M$_X(v)$ of stable sheaves on $X$. The induced automorphism of the function field $K(C)$ will allow us to produce two non-isomorphic derived equivalent K3 surfaces over $\IC$. Since this automorphism of $K(C)$ extends to an automorphism of $\IC$ these two K3 surfaces will also be conjugate.\newline\noindent
The note is organized as follows: First, we consider moduli spaces of K3 surfaces over $\overline{\IQ}$. Then we define Mukai involutions on these moduli spaces and study their fixed point locus in a special case. This information enables us to construct the above mentioned curve. We also show how this result can be used to produce higher-dimensional examples involving Hilbert schemes of points on K3 surfaces. In the last section it is shown that a similar construction method works, and a similar result holds, for abelian surfaces.\newline
A small appendix provides an alternative proof of one of the results in section 3.
\smallskip

\noindent{\bf Acknowledgements.} This paper is a part of my PhD thesis \cite{Sosna} supervised by Daniel \linebreak Huybrechts whom I would like to thank for a lot of fruitful discussions. I am also thankful to the referee for valuable suggestions. Further I would like to thank Heinrich Hartmann for suggesting the proof of Lemma 6.3.
   
\section{K3 surfaces over $\overline{\IQ}$ and their moduli spaces}
Let $\IK$ be an algebraically closed field of characteristic zero. A \emph{K3 surface} is a smooth two-dimensional projective variety $X$ over $\IK$ with trivial canonical bundle and $H^1(X, \ko_X)=0$. An example of such a surface is given by a smooth quartic in $\IP^3_\IK$, e.g.\,the Fermat quartic $\left\{ \underline{x} \in \IP^3_\IK \; : x_0^4+\ldots +x_3^4=0 \right\}$.

An ample line bundle $L$ on $X$ will be called a \emph{polarisation}. The self-intersection number $(L,L)$ of such an $L$ is called the \emph{degree} of the polarised K3 surface $(X,L)$.  A polarisation is \emph{primitive} if $L$ is not a power of any line bundle on $X$. Using Serre duality and the Riemann--Roch theorem, it is easy to see that $(L,L)$ is even for any line bundle $L$ on $X$. Further, if $L$ is ample it is effective, its Hilbert polynomial is given by $h_L(t)=(\frac{1}{2}(L,L)) t^2+2$, $L^n$ is generated by global sections for $n\geq 2$ and is very ample for $n\geq 3$ (for the latter results see \cite{Saint}).

A \emph{family of K3 surfaces} is a proper and flat morphism $\pi: \kx \rightarrow S$ over a scheme $S$ (where here and in the following `scheme' means a scheme of finite type over the field in question), such that the geometric fibres of $\pi$ are K3 surfaces. A family of (primitively) polarised K3 surfaces is given by a map $\pi$ as above together with a line bundle $\kl$ on $\kx$ which defines a (primitive) polarisation restricted to each geometric fibre. Note that $\kl^n$ is relatively very ample over $S$ for $n\geq 3$.

Using the general theory in \cite{Viehweg} it is possible to construct a coarse quasi-projective moduli scheme $\km^\IK_{2d}$ for primitively polarised K3 surfaces of degree $2d$ over $\IK$ as follows: Consider the Hilbert scheme Hilb$_N^{P}$ representing subvarieties of $\IP^N_\IK$ with Hilbert polynomial $P(x):=n^2dx^2+2$, where $n\geq 3$, $d$ is a natural number (which should be thought of as $(\frac{1}{2}\kl^2)$) and $N=P(1)-1$. Then there exists on open subscheme $U$ of Hilb$^P_N$ representing primitively polarised K3 surfaces together with an embedding into $\IP^N_\IK$ and the GIT quotient $U /PGL(N+1) $ is the coarse moduli scheme in question.

For $\IK=\IC$ there is a different construction (which uses lattice and Hodge theory), which shows that the moduli space $\km_{2d}^\IC$ of polarised complex K3 surfaces is actually 19-dimensional, reduced, irreducible and normal. For details see \cite{K3}.\medskip

We would like to see that $\km^{\overline{\IQ}}_{2d}$ inherits all the above mentioned properties. A first step is the following

\begin{lem}
$\km_{2d}^\IC \cong \km^{\overline{\IQ}}_{2d} \times_{\Spec({\overline{\IQ}})} \Spec(\IC)$.
\end{lem} 

\begin{proof}
Both schemes are constructed as GIT-quotients and this is compatible with field extensions, cf. \cite[Prop. 1.14]{GIT}.
\end{proof}

Using the lemma we can now deduce

\begin{prop}
$\km^{\overline{\IQ}}_{2d}$ is an integral 19-dimensional scheme.
\end{prop}

\begin{proof}
Since $\km^{\overline{\IQ}}_{2d} \times_{\Spec({\overline{\IQ}})} \Spec(\IC)$ is irreducible, so is $\km^{\overline{\IQ}}_{2d}$ being its image under the projection. Further, $\km^{\overline{\IQ}}_{2d}$ is reduced, since $\km_{2d}^\IC$ is and we can check this property using an affine cover. The statement about the dimension is clear. 
\end{proof}

We would also like to control the singularities of $\km^{\overline{\IQ}}_{2d}$. This is given by

\begin{prop}
$\km^{\overline{\IQ}}_{2d}$ is a normal scheme.
\end{prop}

\begin{proof}
We will use the following result from commutative algebra: \newline\noindent
Let $R$ and $S$ be two $\IK$-algebras such that $R\otimes_\IK S$ is Noetherian. Then $R\otimes_\IK S$ is normal if and only if $R$ and $S$ are normal (this is a special case of \cite[Thm.\ 1.6]{TY}).\smallskip

\noindent
Since normality is a local property, we may take an open affine subset $\Spec(A)$ in $\km^{\overline{\IQ}}_{2d}$. Then, by normality of $\km^\IC_{2d}$, we have that $A\otimes_{\overline{\IQ}} \IC$ is normal (and of course Noetherian) and hence $A$ is.
\end{proof}

\section{Mukai involutions}

In this section we will study so called Mukai involutions on $\km^{\overline{\IQ}}_{2d}$. We will use several results obtained in \cite{Stellari} where the action of the group of Mukai involutions on $\km^\IC_{2d}$ was investigated. In general, a good reference for results on moduli spaces of sheaves is \cite{HL}.\smallskip 
   
For a complex $2d$-polarised K3 surface $(X,L)$ consider the moduli space $M_X=M_X(v)$ of stable sheaves on $X$ whose numerical invariants are given by a Mukai vector $v \in H^0(X,\IZ)\oplus H^2(X,\IZ)\oplus H^4(X,\IZ)$. We will only consider vectors of the form $v=(r,L,s)$ fulfilling \linebreak (a) $L^2=2rs=2d$, (b) gcd$(r,s)=1$ and (c) $r\leq s$. These choices ensure that $M_X$ is a fine moduli space and a $2d$-polarised K3 surface. We will use the same notation over $\overline{\IQ}$. In this case $M_X$ is also in $\km^{\overline{\IQ}}_{2d}$, which follows from the fact that this is true over $\IC$ and that any line bundle on the K3 surface $(M_X)_\IC$ is already defined over $\overline{\IQ}$. For a proof of the last statement cf.\ \cite[Prop.\ 5.4]{Huy3}. We define the \emph{Mukai involution} $g$ to be the map over $\overline{\IQ}$ sending $X$ to $M_X$. For the discussion of the fact that $g_\IC$, and therefore also $g$, is indeed an involution cf.\ subsection 2.1 in \cite{Stellari}. 

\begin{prop}\label{mor}
The Mukai involution $g$ is a morphism.
\end{prop}  

\begin{proof}
Consider the universal family $f: \kx \rightarrow U$, where $U$ is the above used open subscheme of the Hilbert scheme. Note that $U$ is reduced. The morphism $f$ is projective and hence there exists a relative moduli space $\km(v) \rightarrow U$ such that over $t \in U$ we have the moduli space $M_{\kx_t}(v)$ (compare \cite[Thm.\ 4.3.7]{HL}). By construction there exists a polarisation $\widetilde{\kl}$ on $M_{\kx_t}(v)$. Its intersection number on the fibres is a quadratic multiple, say $a$, of the given degree $2d$. This can be seen by looking at K3 surfaces of Picard rank 1 and using that the intersection number is (locally) constant. Now, the étale sheafification of the relative Picard functor is representable by a scheme $\Pic_{\km(v) / U}$ (see e.g.\ \cite[Ch.\ 8]{BLR}) and the image of the morphism $f: U \rightarrow \Pic_{\km(v) / U}$ defined by $\widetilde{\kl}$ lies in the image of the map $[a]: \Pic_{\km(v) / U} \rightarrow \Pic_{\km(v) / U}$ (where $[a]$ is the multiplication by $a$). We therefore have a commutative diagram
\[\begin{xy}
\xymatrix{ U \ar[rr]^f \ar[rrd]_{\widetilde{f}} & & \Pic_{\km(v) / U}  \\
						& & \Pic_{\km(v) / U} \ar[u]_{[a]}}
  \end{xy}
\]
The morphism $\widetilde{f}$ defines an element of the étale Picard functor, which by definition is represented by a line bundle $\kl'$ on the fibre product $\km(v)':=\km(v)\times_U U' $, for some étale covering $\pi: U' \rightarrow U$, with the property that $\widetilde{\pi}^*(\widetilde{\kl}) \cong \kl'^a$ ($\widetilde{\pi}$ is the natural projection). Thus, $\km(v)' \rightarrow U'$ is a family of K3 surfaces with a polarisation of degree $2d$ and we therefore get a map $\alpha: U'\rightarrow \km^{\overline{\IQ}}_{2d}$. Using descent theory described in \cite[Exp.\ VIII]{SGA} we know that there exists a morphism $\beta: U \rightarrow \km^{\overline{\IQ}}_{2d}$ such that $\beta \pi=\alpha$ if and only if $\alpha$ commutes with the two projections from $U'\times_U U'$, i.e. $\alpha p_1=\alpha p_2$. But the latter condition is clear for closed, and hence for all (everything is reduced), points by the fact that $\alpha$ is the classifying map of the family and the K3 surfaces over the closed points of a fibre of $\pi$ are all isomorphic. Thus we have a map $\beta: U \rightarrow \km^{\overline{\IQ}}_{2d}$
sending $t$ to $M_{\kx_t}(v)$. Since $\alpha$ is equivariant, it descends to a morphism $h$ from the GIT-quotient $U / PGL=\km^{\overline{\IQ}}_{2d}$ to $\km^{\overline{\IQ}}_{2d}$ so that the following diagram commutes
\[\begin{xy}
\xymatrix{ U \ar[r]^\beta \ar[d] & \km^{\overline{\IQ}}_{2d} \\
						\km^{\overline{\IQ}}_{2d} \ar@{.>}[ru]^h&}
  \end{xy}
\]   
By definition we have that $g=h$ and thus $g$ is a morphism.   
\end{proof}

\begin{remark}
Clearly the same proof works for $\km^\IC_{2d}$ (and any other algebraically closed field). In particular, we shall need $g_\IC: \km^\IC_{2d} \rightarrow \km^\IC_{2d}$.
\end{remark}

We can consider a Mukai involution from a different point of view using the following result (see \cite{Mukai}, \cite{Orlov}):

\medskip

\noindent {\bf Theorem (Mukai, Orlov)} Let $X$ and $Y$ be two complex projective K3 surfaces. The following conditions are equivalent:\newline\noindent
(a) $X$ and $Y$ are Fourier--Mukai partners (FM-partners), i.e.\ derived equivalent over $\IC$.\newline\noindent
(b) $Y$ is a fine moduli space of stable sheaves on $X$. 

\medskip

It follows from this theorem that $X_\IC:=X\times_{\Spec{\overline{\IQ}}} \Spec(\IC)$ and $M_X\times_{\Spec{\overline{\IQ}}} \Spec(\IC)=M_{X_\IC}(v)$ are FM--partners. \smallskip

\begin{remark}
If a complex K3 surface $X$ has Picard rank one, then all Fourier--Mukai partners can be determined explicitly by describing the Mukai vectors $v$ as above. Furthermore, $M_X(v)\ncong M_X(v')$ for $v\neq v'$. If $X$ is $2d$-polarised, then the number of FM-partners of $X$ is 
\begin{equation}\label{FM-number}
|FM(X)|=2^{p(d)-1},
\end{equation} 
where $p(1)=1$ and $p(d)$ is the number of prime divisors of $d$ (cf.\ \cite{Oguiso}). 
\end{remark}

From now on we will consider the case $2d=12$. In this case there is only one Mukai involution sending a K3 surface $X$ to the moduli space $M_X=M(2,l,3)$. This involution will be denoted by $g$. The first step is to invesigate its fixed point locus. It was proved in \cite{Stellari} that over $\IC$ there is precisely one divisor $D$ in the fixed point locus of $g_\IC$. We will now prove the

\begin{prop}
The fixed point locus Fix$(g)$ of the morphism $g: \km^{\overline{\IQ}}_{2d} \rightarrow \km^{\overline{\IQ}}_{2d}$ contains a divisor. 
\end{prop}   

\begin{proof}
The fixed point locus of a morphism can be defined as the intersection of the diagonal and the graph which are both defined over $\overline{\IQ}$. Since this construction commutes with base change we have that Fix$(g_\IC)=$Fix$(g)\times_{\Spec(\overline{\IQ})} \Spec(\IC)$. Therefore Fix$(g)$ has to contain a divisor. 
\end{proof}

\begin{remark}
The divisor in Fix$(g_\IC)$ corresponds to K3 surfaces whose Picard lattice contains a certain rank two nondegenerate even lattice. From \cite[Cor.\ 2.5 and Cor.\ 1.9]{Morrison} we know that there exists a K3 surface $X$ of Picard rank 20 in $D$. By \cite[Thm.\ 6]{Shioda} $X$ can be defined over a number field and in particular over $\overline{\IQ}$. In fact, the points of $D$ defined over $\overline{\IQ}$ are dense in $D$ which gives another proof of the above proposition. 
\end{remark}

\section{A special curve in $\km^{\overline{\IQ}}_{2d}$}

We want to construct an irreducible $g$-invariant curve in $\km^{\overline{\IQ}}_{2d}$. The strategy is rather simple: Take a curve in the quotient space $\km^{\overline{\IQ}}_{2d} / \left\langle \id, g \right\rangle$ and pull it back to $\km^{\overline{\IQ}}_{2d}$. The curve $C$ we get will clearly be $g$-invariant. The irreducibility will be achieved by using Bertini's theorem.\newline
We first recall the following 

\begin{prop}
The quotient $\kk=\km^{\overline{\IQ}}_{2d} / \left\langle \id, g \right\rangle$ is an algebraic variety. The projection map $\pi:\km^{\overline{\IQ}}_{2d} \rightarrow \kk$ is finite and surjective.
\end{prop}

\begin{proof}
Since $\km^{\overline{\IQ}}_{2d}$ is quasi-projective, this follows from \cite{Mumford}, pp. 66--69.
\end{proof}

To ensure that $C$ is irreducible, we want it to be connected and regular. The first property is proved by the following  

\begin{prop}
Let $A$ be an irreducible curve in $\kk$ which is not contained in the image of the fixed point locus of $g$ but intersects it in at least one point. Then the pullback curve $C=\pi^{-1}(A)=\km^{\overline{\IQ}}_{2d}\times_\kk A$ is connected. Furthermore $g$ acts non-trivially on $C$. 
\end{prop}  

\begin{proof}
Assume the converse, then there exist disjoint closed non-empty subsets $C=W_1 \bigsqcup W_2$. Since $A$ intersects the fixed point locus, there exists a point $x \in A$ whose reduced fibre is precisely one point $y$. We may assume that $y \in W_1$. Since $\pi$ is finite and thus proper, $\pi(W_1), \pi(W_2)$ are closed in $A$. Since $A$ is irreducible and both sets are non-empty, we must have $\pi(W_1)=\pi(W_2)$. However, this is impossible, since $\pi(W_2)$ does not contain $x$. The last assertion is obvious.
\end{proof}

Next we have to make sure that $C$ is regular. To do this we use Bertini's theorem: Let $A$ be given as an intersection of hyperplanes. If these hyperplanes are generic, then $A$ is regular away from the singularities of $\kk$. In order to control these we need the following 

\begin{lem}
Let $R$ be a normal integral domain, let $H$ be a finite group of automorphisms of $R$ and let $R^H$ be the ring of invariants. Then $R^H$ is normal.
\end{lem}

\begin{proof}
If $z\in Q(R^H)$ is integral over $R^H$, it is also integral over $R$ and hence $z \in R$, since $R$ is normal. On the other hand, $h(z)=z$ for all $h \in H$ and therefore $z \in R^H$. Thus, $R^H$ is normal.
\end{proof}

Since we know $\km^{\overline{\IQ}}_{2d}$ to be normal, it follows from the lemma that $\kk$ is normal as well. Thus, if $A$ is generic among those curves intersecting the fixed point locus in a generic point of the divisor, it will be regular and the same will hold for $C$ (cf. \cite[III, Cor.\ 10.9]{H}). Thus, we have proved

\begin{prop}
There exists a $g$-invariant connected and regular, hence irreducible, curve $C$ (on which $g$ acts non-trivially) in $\km^{\overline{\IQ}}_{2d}$.\qqed
\end{prop}

However, $\km^{\overline{\IQ}}_{2d}$ is just a coarse moduli space and thus a priori we do not have a family over $C$. This problem is avoided by the following

\begin{prop}
There exists a family $\kx' \rightarrow C'$ over an irreducible curve $C'$ such that the classifying map of this family is a finite surjective morphism $C' \rightarrow C$. Denoting the function fields $K(C)$ resp.\ $K(C')$ by $\IK$ resp.\ $\IL$ we furthermore have that the inclusion $\IK \rightarrow \IL$ is a Galois extension.
\end{prop}

\begin{proof}
The first statement is well-known, cf.\ e.g.\ \cite[Thm.\ 9.25]{Viehweg} (the general idea is that the wanted curve lives in the Hilbert scheme). Denote the covering curve by $\widetilde{C}$. Considering function fields we have a finite (and of course separable) extension $\IK=K(C) \rightarrow \IL':=K(\widetilde{C})$. If this extension is not normal, we can take the normal closure $\IL$ of $\IL'$ to get a Galois extension $\IK \rightarrow \IL$. Geometrically this just corresponds to a finite surjective morphism from a new curve $C' \rightarrow \widetilde{C}$ and a (pullback-)family over $C'$. Hence the result. 
\end{proof}

The morphism $g$ induces an automorphism of $\IK$ which can be lifted to an automorphism $\widetilde{g}$ of $\overline{\IK}=\overline{\IL}$ (see e.g.\ \cite[Thm.\ 6]{Yale}). Considering the composition of $\widetilde{g}$ with the inclusion $\IL \rightarrow \overline{\IK}$ and using the 
normality of $\IL$, we see that we get an automorphism $g'$ of $\IL$ which clearly extends $g$. This automorphism then gives an automorphism of the curve $C'$. Since $g'$ is an extension of $g$, it has the same geometric interpretation, namely sending a fibre $X$ to a K3 surface which is isomorphic to $M_X(v)$.

The geometric fibre of the generic point of $C'$ is a K3 surface and therefore the generic fibre itself is as well (use \cite[III, Prop.\ 9.3]{H}). Denote this K3 surface over $\IL$ as $X_\IL$.    
Base change via $g'$ gives a second K3 surface over $\IL$ which by construction is M$_{X_{\IL}}(2,l,3)=:X'_\IL$. Now fix an imbedding $i$ of $\IL$ into $\IC$. Denoting the induced action of $g'$ on $\Spec(\IL)$ by $g'^*$, we have the following diagram:
\[\begin{xy}
\xymatrix{ X_\IC \ar[r] \ar[d] & X_\IL \ar[d] & X'_\IL \ar[l] \ar[d]  & X'_\IC \ar[d] \ar[l]\\
	   \Spec(\IC) \ar[r]^{i^*} & \Spec(\IL) & \Spec(\IL) \ar[l]_{g'^*}  & \Spec(\IC) \ar[l]_{i^*}.}
  \end{xy}
\]  
Extending the automorphism $g'$ of $\IL$ to $\widehat{g} \in \Aut(\IC)$ we see that $X_\IC$ and $X'_\IC$ are conjugate via $\widehat{g}$. Clearly $X'_\IC \cong M_{X_\IC}(2,l,3)$ and thus $X'_\IC$ is not isomorphic to $X_\IC$ for a generic (i.e.\ of Picard rank 1) $X_\IC$. Hence $X_\IC$ and $X'_\IC$ are Fourier--Mukai partners and therefore $X_\IC$ and $\widehat{g}$ are as in the theorem, which is therefore proved.

\begin{remark}
One might ask whether the phenomenon described above is at all special, since a priori it might be that any Fourier--Mukai partner of a complex K3 surface $X$ is obtained by conjugation and/or vice versa. However, this is not the case as illustrated by the following: A generic quartic $X$ in $\IP^3_\IC$ given by a polynomial $P$ has Picard rank 1 and no non-trivial Fourier--Mukai partners by formula \ref{FM-number}. However, conjugating $X$ with a generic automorphism $\sigma$ of $\IC$ will produce a non-isomorphic conjugate quartic $X^\sigma$, since conjugation in this case amounts to applying $\sigma$ to the coefficients of $P$. Hence a non-isomorphic conjugate of a K3 surface need not be a Fourier--Mukai partner.

It should also be true that there exists a generic complex K3 surface having a Fourier--Mukai partner which is not a conjugate surface (for Picard rank 2 there are examples of derived equivalent K3 surfaces with non-isometric Picard lattices, see \cite{Stellari2}, so they cannot be conjugate). By formula \ref{FM-number} it would e.g.\ be enough to have a K3 surface defined over $\IQ$ (so it remains fixed under conjugation) which has Picard rank 1 and is of degree $2d$ with $d$ having several prime divisors. In \cite{vanL} the author gives an explicit example of a quartic K3 surface with Picard rank 1 which is defined over $\IQ$. Hopefully the techniques in \cite{vanL} can be used more generally to produce a K3 surface of sufficiently high degree.
\end{remark}

We can use the result proven above to produce higher-dimensional examples of derived equivalent conjugate varieties. To do this recall that given a projective scheme $W$ over $\IC$ with an embedding $W \subset \IP_\IC^r$, a numerical polynomial $P(t) \in \IQ[t]$ and a scheme $S$ over $\IC$ one defines 
\[
\underline{\Hilb}^W_{P(t)}(S)=\left\{ Z \subset Y\times S \; | \; Z \; \text{proper and flat}/ S,  \; P(Z_s)=P \; \forall \: s \in S \right\}.
\]
This gives a contravariant functor from the category of schemes over $\IC$ to the category of sets. This functor is representable by a projective scheme, cf.\ \cite[Ch.\ 2]{HL}. If $P(t)=n$ is constant, then we will denote the functor by $\underline{\Hilb}^n_W$ and the scheme representing it, the \emph{Hilbert scheme}, by $\Hilb^n(W)$ (of course, this construction works over other fields than $\IC$).  

For a K3 surface $X$ the Hilbert schemes $X^{[n]}:=\Hilb^n(X)$ are so called \emph{irreducible holomorphic symplectic varieties} or \emph{hyperk\"ahler varieties}. This, for us, means that $X^{[n]}$ is a simply-connected projective variety such that the space of global sections of $\Omega^2_{X^{[n]}}$ is generated by a closed non-degenerate holomorphic two-form. The dimension of $X^{[n]}$ is $2n$. We will need the following easy 

\begin{lem}
Let $X$ be a complex projective variety. Then we have the equality $(X^{[n]})^\sigma=(X^\sigma)^{[n]}$ for any automorphism $\sigma$ of $\IC$.
\end{lem}

\begin{proof}
First, note that the map $f \mapsto f^\sigma$ defines an isomorphism $\Hom(S,T)\cong\Hom(S^\sigma,T^\sigma)$ for schemes $S, T$ over $\IC$ and $\sigma \in \Aut(\IC)$, and, writing $\tau$ for $\sigma^{-1}$, the equality $(S^\sigma)^{\tau}\cong S$ holds. Furthermore we also have $(S\times T)^\sigma \cong S^\sigma\times T^\sigma$. The claimed equality now follows from
\[\Hom(S, \Hilb^n(X^\sigma))=\left\{ Z \subset S\times X^\sigma \, | \, Z \, \text{proper}\, \& \, \text{flat}/ S,  \, P(Z_s)=n \: \forall \: s \in S \right\}\cong\]
\[\cong\left\{ Z^\tau \subset S^\tau \times X \; | \; Z^\tau \; \text{proper}\, \& \, \text{flat}/ S^\tau,  \; P(Z^\tau_s)=n \; \forall \: s \in S \right\}=\]
\[=\Hom(S^\tau, \Hilb^n(X))\cong\Hom(S, \Hilb^n(X)^\sigma),\]
since an object representing a functor is unique up to isomorphism.
\end{proof}

Now recall from \cite[Prop.\ 8]{Ploog} that if we have an equivalence $\Db(X)\cong \Db(Y)$ of two smooth projective surfaces $X$ and $Y$, then there is also an equivalence $\Db(\Hilb^n(X))\cong \Db(\Hilb^n(Y))$. Thus, we derive the

\begin{thm}\label{conj-ihs}
There exist a hyperk\"ahler variety $Y$ and an automorphism $\sigma$ of $\IC$ such that $Y$ and $Y^\sigma$ are derived equivalent. 
\end{thm}

\begin{proof}
Set $Y=\Hilb^n(X)$ with $X$ as in the main theorem and use the lemma.
\end{proof}

\begin{remark}
If $X$ and $Y$ are non-isomorphic K3 surfaces, then it is not true in general that $\Hilb^n(X)$ and $\Hilb^n(Y)$ are also non-isomorphic for all $n$, cf.\ \cite[Ex.\ 7.2]{Yoshioka} where the author gives an example such that the Hilbert schemes are indeed isomorphic for $n=2$ and $n=3$. Thus, in the above theorem we might have produced isomorphic Fourier--Mukai partners. However, we believe that the examples in \cite{Yoshioka} are rather special and that the theorem will indeed produce non-isomorphic Hilbert schemes in the general case. 
\end{remark}

\section{Abelian surfaces}

We will now show that a similar result holds for abelian surfaces. First, recall that for any algebraically closed field $\IK$ of characteristic $0$ there exists a coarse moduli space $\ka_{(1, t)}$ for polarised abelian surfaces of type $(1, t)$. This moduli space is a quasi-projective normal threefold  (see e.g.\ \cite{GIT}).

Using \cite{BH2} we know that there exists an involution $f$ on $\ka_{(1, t)}$ sending a polarised abelian surface $(A, L)$ to the dual polarised abelian surface $(\widehat{A}, \widehat{L})$ (thus Prop. \ref{mor} is known for abelian surfaces). Here, the dual polarisation is e.g.\ defined by demanding that its pullback under the isogeny induced by $L$ is $tL$. For other descriptions of $\widehat{L}$ (in the case where the ground field is $\IC$) see \cite{BH1}.

In the case $\IK=\IC$ we use \cite[Thm.\ 3.4]{GH} to deduce that the fixed point locus of $f_\IC$ contains at least one divisor. Same arguments as in the K3 case show that Fix$(f)$ contains a divisor as well. We can then proceed as in section 4 to produce a certain curve in $\ka_{(1, t)}$. To conclude we just note that by the classical results of Mukai, see \cite{Muk-ab}, $A$ and $\widehat{A}$ are always derived equivalent. 

\begin{remark}
Similarly to the case of K3 surfaces one can associate a hyperk\"ahler variety $K_n(A)$, $n\geq 1$, to any abelian surface $A$. It is of dimension $2n$ and is called the \emph{generalised Kummer variety} of $A$. The name stems from the equality $K_1(A)={\rm{Kum}}(A)$. One could hope to use the above theorem to produce yet another example of derived equivalent conjugate hyperk\"ahler varieties. However, in the K3-case we used the fact that derived equivalent K3 surfaces have derived equivalent Hilbert schemes. Unfortunately, a similar result is missing for the generalised Kummer varieties. In \cite{Namikawa} the author gives an example which shows that the generalised Kummer varieties of derived equivalent abelian surfaces need not be birational. It was shown in \cite{HLOY2} that for any two abelian surfaces $A$ and $B$ one has
\[ \Db(A) \cong \Db(B) \Longleftrightarrow \Db({\rm{Kum}}(A))\cong \Db({\rm{Kum}}(B)) \Longleftrightarrow {\rm{Kum}}(A) \cong {\rm{Kum}}(B).\]
It is open what kind of statement one has in higher dimensions.  
\end{remark}

\section{Appendix}
In this section we will give a different proof of Proposition \ref{mor}. The strategy is as follows: First prove that any Mukai involution is an analytic morphism of the moduli space over $\IC$, then show that it is in fact algebraic and using this conclude that a Mukai involution over $\overline{\IQ}$ is a morphism as well. We will use the usual notations: In particular, $\Lambda=U^{\oplus 3}\oplus E_8(-1)^{\oplus 2}\cong H^2(X,\IZ)$ (where $X$ is a complex K3 surface) is the K3 lattice, $\widetilde{H}(X,\IZ)$ is the integral cohomology of $X$ endowed with the Mukai pairing, $\Omega=\left\{[x] \in \IP(\Lambda_\IC)\; | \; x^2=0, x\overline{x}>0\right\}$ is the period domain, $\Lambda_{2d}=\left\langle k\right\rangle \oplus U^{\oplus 2}\oplus E_8(-1)^{\oplus 2}$, where $k^2=-2d$, is the orthogonal complement of a vector $h \in \Lambda$ with $h^2=2d$ (using this we can then define $\Omega_{2d}$), $\Gamma(h)=\left\{ g \in O(\Lambda) \; | \; g(h)=h\right\}$ and finally $\Gamma_{2d}=\im (\Gamma(h) \rightarrow O(\Lambda_{2d}))$.   

\begin{prop}
Any Mukai involution is an automorphism of $\km^\IC_{2d}$ considered as a complex-analytic space.
\end{prop}

\begin{proof}
Let $g_\IC$ be an arbitrary Mukai involution sending $X$ to $M_X(r,l,s)=:M(v)$. We know from \cite{Mukai} that there exists a Hodge isometry $\alpha: \Lambda \cong v^\bot / \IZ v$ (where $v^\bot$ is the orthogonal complement of $v$ in the Mukai lattice $\widetilde{H}(X,\IZ)$). It is clear that $\alpha$ induces an isomorphism $\IP(\Lambda\otimes \IC) \cong \IP((v^\bot / \IZ v)\otimes \IC)$. This isomorphism restricts to an isomorphism of the period domains $\Omega$ and $\Omega'$ on both sides. 

Now recall that if $X$ is a K3 surface, then we get its associated period point $x_0\in \Omega$ by choosing an isometry $\varphi: H^2(X, \IZ) \cong \Lambda$. We get the period point $\widetilde{x}_0$ of $M(v)$ in $\Omega$ by using the isometry $\alpha^{-1}\circ \psi^{-1}: H^2(M(v), \IZ) \cong \Lambda$ where $\begin{xy}\xymatrix{\psi: v^\bot / \IZ v \ar[r]^{\cong} & H^2(M(v), \IZ)}\end{xy}$ is the isometry from \cite{Mukai}. The map $g_\IC$ sends $x_0$ to $\widetilde{x}_0$. Now, $\alpha(x_0)$ is by definition the period point of $M(v)$ considered as a point in $\Omega'$. Thus, on the level of period domains the Mukai involution $g_\IC$ is just the isomorphism $\alpha$. Factoring out the markings is compatible with this process. The same proof works in the polarised setting. 
\end{proof}

\begin{prop}
A Mukai involution $g_\IC$ is an algebraic morphism.
\end{prop}

\begin{proof}
We want to apply the following theorem of Borel (see \cite{Borel}):

\noindent{\textit{If $Y$ is a quasi-projective variety and $f: Y \rightarrow \Omega / \Gamma$ is a holomorphic map to the quotient of a homogeneous symmetric domain $\Omega$ by an arithmetic torsion-free group $\Gamma$, then $f$ is algebraic.}}

We wish to apply the theorem to $Y=\km^\IC_{2d}$, $\Omega=\Omega_{2d}$, $\Gamma=\Gamma_{2d}$ and $f=g_\IC$. Since $\Gamma_{2d}$ is in general not torsion-free, we cannot apply the theorem directly. However, using results in \cite{Popp} (compare also the proof of Prop.\ 2.2.2 in \cite{Hassett}) we can avoid this problem by using level covers and considering the algebraic construction of $\km^\IC_{2d}$ described in Section 2. Denote the open subset of the Hilbert scheme used there by $U$. 

For $l\geq 3$ consider $\Gamma_{2d}(l)$, the $l$-th congruence subgroup of $\Gamma_{2d}$. This group is torsion-free and the projection $\Omega_{2d} / \Gamma_{2d}(l) \rightarrow \Omega_{2d} / \Gamma_{2d}$ is finite. Since the group of automorphisms of a polarised K3 surface that fix the polarisation is finite and this group acts faithfully on the cohomology of a K3 surface, we can apply \cite[Prop.\ 2.17]{Popp}, which gives us a finite Galois covering $U'$ of $U$ such that $U' \rightarrow U$ has Galois group $\Gamma_{2d} / \Gamma_{2d}(l)$. Thus we get a commutative diagram
\[
\xymatrix{ U' / PGL \ar[d] \ar[r] & \Omega_{2d} / \Gamma_{2d}(l) \ar[d] \\
				\km^\IC_{2d} \ar[r]^{g_\IC} & \Omega_{2d} / \Gamma_{2d}.}
\]
All the varieties in the diagram are quasi-projective and the vertical arrows are finite and surjective maps given as quotients of the action of a finite group. By Borel's result the upper map is algebraic and thus so is $g_\IC$. 
\end{proof}

To conclude one has to check that $g$ is also a morphism. Since we are working over algebraically closed fields we can switch to the classical language and only consider closed points. So, $g$ is a map of sets and by the above $g_\IC$ is a morphism, i.e. locally on affine sets given by polynomials which a priori could have non-algebraic coefficients. However, we know that applied to points with $\overline{\IQ}$-entries these polynomials give algebraic values. 

\begin{lem}\label{alg-val}
Let $X \subset \IA^n_{\overline{\IQ}}$ be an affine variety with coordinate ring $K[X]$ and let $p \in K[X_\IC]$ be a function having algebraic values on $X$. Then $p \in K[X]$.
\end{lem}

\begin{proof}
Set 
\[H:=\left\{\sigma \in \Aut(\IC) \; : \; \sigma_{|\overline{\IQ}}=\id \right\},\] 
then $\IC^H:=\left\{ c \in \IC \; : \; \sigma(c)=c \ \forall \,  \sigma \in H \right\}=\overline{\IQ}$. It follows that $\IC[x_1,\ldots, x_n]^H=\overline{\IQ}[x_1,\ldots, x_n]$ and $K[X_\IC]^H=K[X]$.
Now consider $p-hp$ for a polynomial $p$ as above and $h \in H$. By assumption we have that $p-hp$ is zero on $X_\IC(\overline{\IQ})$ and therefore also on $X_\IC(\IC)$ since $X_\IC(\overline{\IQ})\subset X_\IC(\IC)$ is dense. Therefore $p-hp=0 \in K[X_\IC]$ and thus $p \in K[X]$. 
\end{proof} 

Thus we have re-proved the

\begin{prop}
The map $g$ is a morphism.\qqed
\end{prop}

\end{document}